\documentclass[12pt,letterpaper, reqno]{amsart}
\usepackage{amssymb,latexsym, amsmath, amsxtra}
\usepackage[dvips]{graphics}

\theoremstyle{plain}
\newtheorem{theorem}{Theorem}[section]
\newtheorem*{nonumtheorem}{Theorem}
\newtheorem{corollary}[theorem]{Corollary}

\newtheorem{lemma}[theorem]{Lemma}

\theoremstyle{definition}

\theoremstyle{remark}

\numberwithin{equation}{section}
\numberwithin{theorem}{section}
\numberwithin{table}{section}
\numberwithin{figure}{section}

\newcommand{\C}{\mathbb C}
\newcommand{\N}{\mathbb N}
\newcommand{\R}{\mathbb R}
\newcommand{\Z}{\mathbb Z}

\renewcommand{\H}{\mathcal H}

\def\({\left(}
\def\){\right)}
\def\th{\theta}
\def\cI{\mathcal I}
\def\SL{\operatorname{SL}}
\def\PSL{\operatorname{PSL}}

\begin{document}
\title[Zeros of period polynomials lie on the unit circle]{The nontrivial zeros of period polynomials of modular forms lie on the unit circle}
\author{J. Brian Conrey, David W. Farmer, and \"Ozlem Imamo\-glu}

\thanks{
Research of the first two authors supported by the
American Institute of Mathematics
and the National Science Foundation, the research of the third author is supported by Swiss National Science Foundation}

\thispagestyle{empty}
\vspace{.5cm}
\begin{abstract}
We show that all but 5 of the zeros of the period polynomial associated to a Hecke cusp form are on the unit circle.
\end{abstract}

\address{
{\parskip 0pt
American Institute of Mathematics\endgraf
conrey@aimath.org\endgraf
farmer@aimath.org\endgraf
\null
Department of Mathematics\endgraf
ETH Zurich\endgraf
ozlem@math.ethz.ch
}
  }

\maketitle

\section{Introduction}

Let ${\mathcal M}_k(\Gamma)$ be the space of holomorphic modular forms of weight $k$ for the full modular group $\Gamma=\PSL(2,\Z)$. It is well known that ${\mathcal M}_k(\Gamma)$ has dimension $\frac{k}{12} +O(1)$ and a modular form $f\in {\mathcal M}_k$ has $\frac{k}{12} +O(1)$ inequivalent zeros in  a fundamental domain $\Gamma\backslash\H.$
The study of the natural question of the  distribution of the zeros of modular forms  dates back to the 1960's and has seen some renewed interest thanks to the recent progress on the Quantum Unique Ergodicity (QUE) conjecture.  

In the simplest case   of  Eisenstein series, it was conjectured by R.A. Rankin in 1968 and proved by F.K.C. Rankin and Swinnerton-Dyer \cite{RS-D} that all the zeros, in the standard fundamental domain, of  the series
$$E_k(z)=\frac{1}{2}\sum_{(c,d)=1} (cz+d)^{-k}$$
lie on the geodesic  arc $\{z\in\H \ : \ |z|=1, 0\leq\Re{z}\leq 1/2\}$
and as $k\rightarrow\infty$  they become uniformly distributed on this unit arc. A similar result for the cuspidal Poincare series was proved by R.A. Rankin \cite{R}. For generalizations of these results to other Fuchsian groups and  to weakly holomorphic modular functions   see   \cite{AKN}, \cite{DJ}, \cite{H}, among many others.

In contrast to these cases, for the cuspidal Hecke eigenforms, it is   a consequence of the recent proof  
  of the holomorphic Quantum Unique Ergodicity(QUE) conjecture by  Holowinsky and Soundararajan  \cite{HS}  that the zeros are uniformly distributed.
More precisely, we have 
\begin{nonumtheorem}(Holowinsky and Soundararajan \cite{HS})
Let $\{f_k\}$ be  a sequence of cuspidal Hecke eigenforms.
Then as $k\rightarrow\infty$ the zeros of $f_k$ become 
equidistributed with respect to the normalized hyperbolic measure $\frac{3}{\pi}\frac{ dx dy}{ y^2}$
\end{nonumtheorem}  

For some recent work on the zeros of holomorphic Hecke cusp forms that lie on the geodesic segments of the standard fundamental domain see \cite{GS}.

In this note we turn our attention from the zeros of modular forms to the zeros of their  period polynomials.   

It is well known that  $\Gamma$ is generated by the elliptic
transformations $S=\pm \left(\begin{smallmatrix}
                 0 & 1\\
                 -1 & 0
\end{smallmatrix}\right)$ and $U=\pm \left(\begin{smallmatrix}
                 1 & -1\\
                 1 & 0
\end{smallmatrix}\right)$
    with the defining relations $S^2=U^3=\pm1.$
  
Let   $P_{k-2}$ be the space of all complex polynomials of degree at most $ k-2$.
For $p(z) \in P_{k-2}$,    $A\in \PSL(2,\C)$ acts on $p(z)$ in the usual way via 
 $$ (p\vert A)(z) :=(cz+d)^{k-2}p\left(\frac{az+b}{cz+d}\right)$$       
Let  $P^-_{k-2}$ be the space of odd polynomials of degree $k-2$ and 
$$
 W^-= W^-_{k-2}=\{p \in P^-_{k-2}; p\vert(1+S)=
p\vert(1+U+U^2)=0\}.$$

For $f(z)=\sum_{n=1}^\infty a(n)e^{2\pi inz}$  a Hecke eigenform of even integral weight $k =w+2$ and level~1,  let  $L_f(s)=\sum_{n=1}^\infty a(n)n^{-s}$
be its associated $L$-function. 
The odd period polynomial for $f$ is defined by
\begin{equation}
r_f^-(X):= \sum_{n=1\atop n \mathrm{odd}}^{w-1}(-1)^{\frac{n-1}2}{w\choose n} n! (2\pi)^{-n-1}L_f(n+1)X^{w-n}.
\end{equation}

The basic result of Eichler Shimura theory is

\begin{nonumtheorem}[Eichler-Shimura] Let $\mathcal S_k(\Gamma)$ be the space of cusp form for $\Gamma$. Then the map 

\begin{align*}r^-: \mathcal S_k(\Gamma)&\rightarrow W^-\\
f &\rightarrow r_f^-(X)
\end{align*} 
 is an isomorphism. 
 
\end{nonumtheorem}  


In the light of the Theorem of Eichler and Shimura, studying the zeros of period polynomials is as natural as studying the zeros of modular forms. In this paper we prove 

\begin{theorem}\label{mainth}
If $f$ is a Hecke eigenform, then the odd period polynomials $r_f^-(X)$ have simple zeros at $0,\pm 2$,and $\pm 1/2$ and double zeros at $\pm 1$. The rest of its zeros are complex
numbers on the unit circle.
\end{theorem}

Figure~\ref{fig:hecke} illustrates Theorem~\ref{mainth} in the case
of $f$ a cusp form of weight $w=34$.  Note that in this 
example the spacing between zeros is quite regular.  
From the proof of the Theorem \ref{mainth} it will become clear that
this is a general phenomenon.
It is worth noting that for an arbitrary cusp form which is not a Hecke eigenform the zeros of $r_f^-(X)$ need not be on the unit circle. This can be thought as analogous to the fact that
  for a general modular form $f$ which is not a Hecke eigenform, the distribution of zeros of $f$ need not be uniformly distributed.

\begin{figure}[htp]
\begin{center}
\scalebox{1.0}[1.0]{\includegraphics{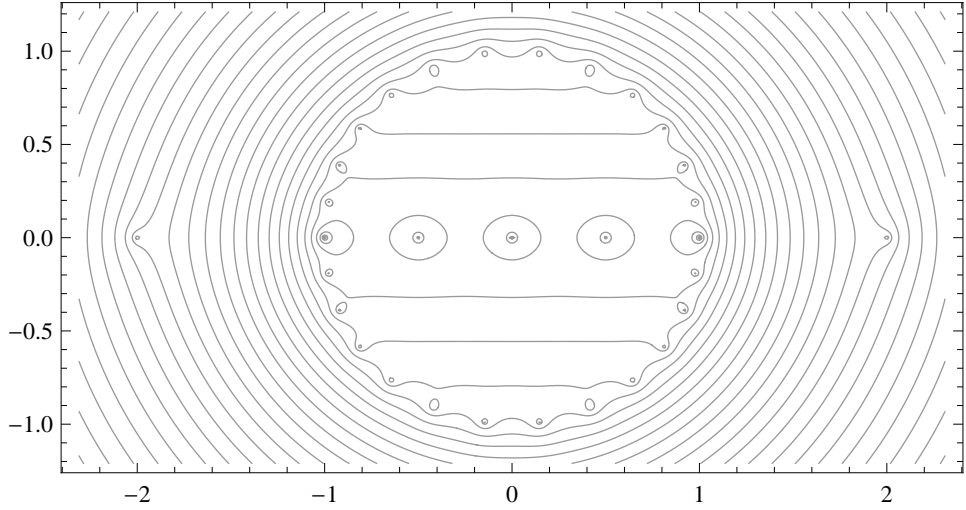}}
\caption{\sf
A contour plot of $\log|r_f^-(z))$ for $f$ one of the cusp forms of weight
$w=34$, illustrating the zeros at $\pm 2$, $\pm \frac12$, and $0$,
with the remaining zeros on the unit circle.
} \label{fig:hecke}
\end{center}
\end{figure}

    In the case of zeros of modular forms, the uniform distribution result is a remarkable consequence of the deep QUE conjecture, which is now a theorem due to Holowinsky and Soundararajan for holomorphic eigenforms. 
    The fact that the uniform distribution of the zeros of modular forms follows from the QUE conjecture 
was first observed by  S. Nonnenmacher and A. Voros \cite{NV}, B. Shiffman and S. Zelditch \cite{SZ} and Z. Rudnick  \cite{Rud}.
 
  In the case of the zeros of period polynomials, as we will show in the next section, Theorem \ref{mainth} follows using simple function theory arguments together with the deep   theorem of Deligne which is the Petersson-Ramanujan conjecture in the case of holomorphic cuspforms.
  
  Finally it is worth noting that the  proof of our Theorem \ref{mainth} can be applied without much difficulty to show that the zeros of       some special period polynomials are also on the unit circle. More precisely these are the polynomials associated to the cusp forms $R_n(z)$, $0\leq n\leq w=k-2$  characterized by the property 
  $$r_n(f):= n!(2\pi)^{n-1}L(f,n+1) = (f, R_n), \,\,\, \forall f\in S_k(\Gamma).$$
  Here  $(f, R_n)$ is the Petersson inner product of $f$ and $R_n$. 
  $R_n(z)$ has the following Poincare type series representation, due to H.Cohen \cite{C}.
  For $0<n<w$,  
 ${\tilde n}= w-n$ and $c_{k,n}=i^{{\tilde n}+1}2^{-w}{w\choose n}\pi$,   we have 
      $$ R_n(z)=c_{k,n}^{-1}\sum_{ \left(\begin{smallmatrix}
                  a & b\\
                 c & d
 \end{smallmatrix}\right)\in\Gamma} (az+b)^{-n-1}(cz+d)^{-{\tilde n}-1}
$$
  
  A special case of   Theorem 1 in \cite{KZ}     gives that   the odd period polynomial of $R_n$ for $n$ even and $0<n<w$ is given by the Bernoulli type polynomial 
  \begin{equation}\label{bernoulli}
  (-1)^{k/2+n/2}2^{-w}r^-_{R_n}(X)= \left[\frac{B^0_{\tilde n+1}(X)}{\tilde n+1}-\frac{B^0_{n+1}(X)}{n+1}\right] \mid (I-S)
  \end{equation}
  where 
  $$B^0_{n+1}(X)=\sum_{\substack{i=0\\ i\neq 1}}^{n+1}{n\choose i}B_iX^{n+1}.$$
  
  The polynomials in (\ref{bernoulli}) can also be  closely approximated by 
   $\sin(2\pi x) +x^N \sin(2\pi/x)$ which then can be used to show that their non-trivial zeros are on the unit circle. The period polynomials of  the cusp forms $R_n$ can be seen as complementary to  the Ramanujan polynomials which can be thought in terms of the period polynomials of the Eisenstein series. (see \cite{KZ} and \cite{F}) Recently, it was shown by R. Murty, C. Smyth, and R. Wang \cite{MSW} that  the zeros of the Ramanujan polynomials   also   lie on the unit circle.  In this context see also \cite{GMR}.

\section{Period polynomial of Hecke eigenforms}

For  $f(z)=\sum_{n=1}^\infty a(n)e^{2\pi inz}\in \mathcal S_k(\Gamma)$,  a Hecke eigenform, we   let  $L_f(s)=\sum_{n=1}^\infty a(n)n^{-s}$
be its associated $L$-function and   
\begin{equation}
r_f^-(X):= \sum_{n=1\atop n \mathrm{odd}}^{w-1}(-1)^{\frac{n-1}2}{w\choose n} n! (2\pi)^{-n-1}L_f(n+1)X^{w-n}.
\end{equation}
the odd part of its period polynomial. 

The L-function satisfies the functional equation
$$
(2\pi)^{-s}\Gamma(s)L_f(s)=(-1)^{k/2}(2\pi)^{s-k}\Gamma(k-s)L_f(k-s).
$$
It follows from the functional equation that  $r_f^-(X)$ is self-reciprocal, i.e.
\begin{eqnarray} \label{eqn:rec}
r_f^-(X) =X^{w}r_f^-(1/X)
\end{eqnarray}
and it follows from the modularity of $f$ (specifically that $f\big(\frac{z-1}{z}\big)=z^kf(z)$) that
\begin{equation}
 r_f^-(X)+X^w r_f^-\big(1-\frac 1X\big) +(X-1)^w r_f^-\big(\frac{-1}{X-1}\big)=0.
 \end{equation}
By Eichler Shimura theory  the vector space   of polynomials of degree less than or equal to $k-2$ spanned by the set of $r_f^-(X)$
as $f$ runs through Hecke eigenforms of weight $k$  is precisely the space of odd polynomials $P$ of degree $\le k-3$
for which
\begin{equation}
\label{eqn:per1} P(x)+x^{k-2}P\big(\frac {-1}x\big)\equiv 0 
\end{equation}
and
\begin{equation}
\label{eqn:per2} 
 P(x)+x^{k-2}P\big(1-\frac 1x\big) +(x-1)^{k-2}P\big(\frac{-1}{x-1}\big)\equiv0.\end{equation}

\begin{lemma} The polynomial $p$ in Theorem \ref{mainth} has
``trivial zeros'' at $\pm 2$, $\pm \frac12$, and $0$.
\end{lemma}

\begin{proof}
Since $P$ is odd, we have $P(0)=0$ and we only have to verify that $P(1)=P'(1)=P(2)=P(1/2)=0$. We substitute $x=1$ into (\ref{eqn:per2}), noting that  
 \begin{eqnarray*}
 \lim_{x\to 1} (x-1)^{k-2}P\big(\frac{-1}{x-1}\big)=0
 \end{eqnarray*}
 since $P$ has degree smaller than $k-2$. Thus, $P(1)=0=P(-1)$.  Now we substitute $x=-1$ into (\ref{eqn:per2}) to obtain
 $$ P(2) +2^{k-2}P(1/2)=0
 $$
 while from $x=2$ in (\ref{eqn:per1}) we have
 $$P(2)-2^{k-2}P(1/2)=0.$$
 Thus, $P(1/2)=P(2)=0$.  We differentiate (\ref{eqn:per1}) to obtain
 \begin{equation*}
 P'(x)+(k-2)x^{k-3}P\big(\frac {-1}x\big)+x^{k-4}P'\big(\frac {-1}x\big)\equiv 0.
\end{equation*}
Substituting $x=1$ here gives
$$ P'(1)+P'(-1)=0.$$ 
 But $P$ is odd, so $P'$ is even which means that $P'(-1)=P'(1)$. Therefore, $P'(1)=P'(-1)=0$ and so we have verified that all of the trivial 
 zeros are where we said they would be. 
\end{proof}

To understand the rest of the zeros, we look at 
\begin{align}p_k(X):=
\frac{2\pi r_f^-(X)}{(-1)^{w/2}(2\pi)^{-w}(w-1)! }= \mathstrut&  \sum_{ n=1 \atop  n \mathrm{odd}}^{w-1}(-1)^{\frac{n-1}2} \frac{(2\pi X)^n}{n!  } L_f(w-n+1) \cr
 = \mathstrut&    \sum_{m=0}^{w/2-1}  \frac{(-1)^{m}(2\pi X)^{2m+1}}{(2m+1)! } L_f(w-2m) .
\end{align}
Since $L_f(w-2m)$ is close to 1 for small values of $m$ we see that  the initial terms of the above are close to the initial terms of the series
 $$\sin(2 \pi X)=\sum_{m=0}^\infty \frac{(-1)^m (2\pi X)^{2m+1}}{(2m+1)!}.$$
 The idea now is to study the zeros of 
 $$\sin(2\pi x) +x^N \sin(2\pi/x).$$
 
It follows from (\ref{eqn:rec}) that $p_f^-(X)$ may be written as
\begin{eqnarray} \label{eqn:q}
p_f^-(X)= q_f(X)+X^{w}q_f(1/X)
\end{eqnarray}
where
\begin{eqnarray} \label{eqn:qf}
q_f(X)= \sum_{m=0}^{[(w-6)/4]}  \frac{(-1)^{m}(2\pi X)^{2m+1}}{(2m+1)! } L_f(w-2m) +\frac{L_f\big(\frac{w+2}{2}\big) (2\pi X)^{\frac{w}2}}{2\big(\frac{w}{2}\big)!}.
\end{eqnarray}
 Note that when $k\equiv 2 \bmod 4$ the last term doesn't appear, since in this case the functional equation implies that $L_f(k/2)=0$. 
 Note also that $q_f$ and $r_f^-$ have real coefficients, since $L_f(s)$ is real on the real axis.
 \bigskip 
 
To prove that the non trivial zeros of $r_f^-(X)$ are on the unit circle we need several lemmas. 
First we   can replace $\sin 2\pi z$ above by an entire function $r(z)$. The crucial idea is the following lemma.

 \begin{lemma} \label{gen-lemma}Let  $r(z)$ be an entire function  and for $N\in \N$ let
 $$F_N(z):=r(z)+z^Nr(z^{-1}).$$
 Let  
 $$R(\th):=\Re{r(e^{i\th})} \qquad \mathrm{and} \qquad I(\th):=\Im{r(e^{i\th})}.$$
  For $j=0, \ldots 2M-1$
 let $\cI_j$ denote the interval $[\frac{\pi}{2M}+\frac{\pi}{M}j, \frac{\pi}{2M}+\frac{\pi}{M}(j+1)]\subset \R$.
 Then 
\begin{enumerate}
\item For  $\th_j=\frac{\pi}{2M}+\frac{\pi}{M}j$,  if $I(\th_j)=0$,  then   $F_{2M}(e^{i\th_j})=0$. 
\item If $I(\th)\neq 0$ for $\th\in \cI_j$, then $F_{2M}((e^{i\th})=0$ for some $\th\in \cI_j$.
\end{enumerate}
   \end{lemma}
   \begin{proof}
   Let
   $$f_N(z)=z^{-N/2}F_N(z)=z^{-M}r(z)+z^Mr(z^{-1}),$$
where $2M=N$. Note that $f_N$ and $F_N$ have the same zeros on $|z|=1$. Since $f_N$ is real when $|z|=1$, it suffices to look at the real valued  function
   \begin{equation}\label{re-part}
   \Re f_N(e^{i\th})= 2\cos(M\th)R(\th)+2\sin(M\th)I(\th)
   \end{equation}
   Using $(\ref{re-part})$,  part $a)$ of the Lemma is clear.

   To see part $b)$ note that If $I(\th)\neq 0$ then $\Re f_N(e^{i\th})=0$ will have a solution when
   \begin{equation}\label{tan-eq}
   -\tan{M\th}=\frac{R(\th)}{I(\th)}
   \end{equation}
   If $I(\th)\neq 0$  for $\th\in\cI_j$ then in the interval $\cI_j$ the function $R(\th)/I(\th)$ is bounded and continuous and hence $(\ref{tan-eq})$ will have a solution.
   \end{proof}

\begin{lemma} \label{lemma:S} Let
$$S(z)=\sin(2\pi z)-\sin(2 \pi /z).$$
Then $S(z)$ has precisely 10 zeros in the annulus $A:=\{z:4/5\le |z|\le 5/4\}$.
Moreover, on the boundary of the annulus,
$|S(z)| > 1.$
\end{lemma}

This lemma can be verified with the aid of a graphing program such
as Mathematica. One can plot the image of $S(z)$ on
the boundary of $A$ and use the argument principle, or view a contour plot
as shown in Figure~\ref{fig:lemma}.

\begin{figure}[htp]
\begin{center}
\scalebox{0.8}[0.8]{\includegraphics{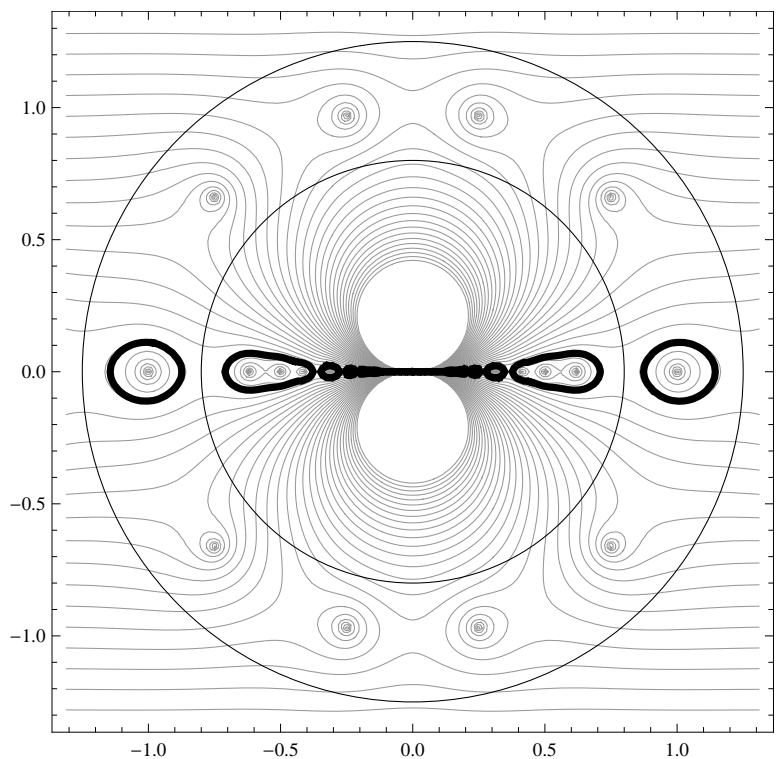}}
\caption{\sf
A contour plot of $\log|S(z)|$.  The darker contour is the set where
$\log|S(z)| = \log(1.5)$.
} \label{fig:lemma}
\end{center}
\end{figure}

\begin{lemma} \label{lemma:est}
Let $f$ be a Hecke eigenform of weight $k$ for the full modular group and let $L_f(s)=\sum a_f(n){n^s}$ be its associated L-function.
 Then
for $\sigma \ge 3k/4$, we have 
\begin{eqnarray}\label{eqn:est1}
|L_f(\sigma)-1|\le 4 \times 2^{-k/4}
\end{eqnarray}
and
for $\sigma $ an integer with $\sigma\ge k/2$, we have 
\begin{eqnarray} \label{eqn:est2}
L_f(\sigma) \le 2k^{1/2} \log 2k+1.
\end{eqnarray}
\end{lemma}
\begin{proof}
If $\sigma \ge 3k/4$ then we are in the region of absolute convergence. By Deligne's Theorem we have 
\begin{eqnarray*} 
|L_f(\sigma)-1|\le \sum_{n=2}^\infty \frac{d(n)}{n^{\sigma-(k-1)/2}}\le \sum_{n=2}^\infty \frac{d(n)}{n^{k/4}}=\zeta(k/4)^2-1\le 2(\zeta(k/4)-1).
\end{eqnarray*}
We have
\begin{eqnarray*}
\zeta(k/4)-1=2^{-k/4}+\sum_{n=3}^\infty n^{-k/4}\le 2^{-k/4}+ \int_{2}^\infty u^{-k/4} ~du \le 2 \times 2^{-k/4} 
\end{eqnarray*}
for $k\ge 12$. This proves (\ref{eqn:est1}).

Next, if $\sigma\ge k/2+1$ we estimate $L_f(\sigma)$ trivially by $L_f(\sigma)\le \zeta(3/2)^2<7$. If $k/2\le \sigma \le k/2+1$, then 
 using standard methods, we have for $\sigma=k/2$,
\begin{eqnarray*}
  \Gamma(k/2) L_f(k/2) = 2  \sum_{n=1}^\infty \frac{a_f(n)}{n^{k/2}}\int_{2\pi n }^\infty e^{-x}x^{k/2} \frac{dx}{x}.
\end{eqnarray*}
Thus
\begin{eqnarray*}
|L_f(\sigma)|\le 
   2\Gamma(k/2)^{-1}\sum_{n=1}^\infty  \frac{d(n)}{n^{1/2}} \int_{2\pi n }^\infty e^{-x}x^{k/2}
 \frac{dx}{x}.
\end{eqnarray*}
We split the sum over $n$ at $k$. The terms with $n\le k$ are
\begin{eqnarray*}
\le 
   \sum_{n\le k}  \frac{d(n)}{n^{1/2}} 
\end{eqnarray*}
as is seen by completing the integrals down to 0.
Now
\begin{eqnarray*}
\sum_{n\le k} \frac{d(n)}{n^{1/2}}=\sum_{mn\le k}\frac{1}{(mn)^{1/2}}\le \sum_{m\le k}\frac{1}{m^{1/2}}\int_0^{k/m}u^{-1/2}~du
=2k^{1/2}\sum_{m\le k}\frac 1 m \le  2k^{1/2} \log 2k
\end{eqnarray*}
 for $k\ge 5$.  The tail of the series is 
 \begin{eqnarray*}&&
 =2\Gamma(k/2)^{-1}\sum_{n=k+1}^\infty\frac{d(n)}{\sqrt{n}}\int_{2\pi n}^\infty e^{-x}x^{k/2}\frac{dx}{x}\\
 &&\le 2\Gamma(k/2)^{-1}\sum_{n=k+1}^\infty\frac{d(n)}{\sqrt{n}} e^{-\pi n} \int_{2\pi n}^\infty e^{-x/2}
 x^{k/2}\frac{dx}{x}.\end{eqnarray*}
The integral is 
$$=2^{k/2} \int_{\pi n}^\infty e^{-x} x^{k/2} \frac{dx}{x}\le 2^{k/2}\Gamma(k/2).$$
Using $d(n)\le 2 \sqrt{n}$ we have that the tail is
$$\le 4\times 2^{k/2}\sum_{n=k+1}^\infty e^{-\pi n}\le 4\times 2^{k/2} e^{-\pi k}<1.$$
Note that $2k^{1/2}\log 2k +1 >7 $ for $k\ge 3$. The proof is complete.
\end{proof}

\begin{lemma} \label{lemma:est2}
Let $q_f$ be as in (\ref{eqn:qf}). Then, for $z \le 5/4$ and $k\ge 80$ we have 
$$|\sin 2 \pi z -q_f(z)|\le \frac{1}{100}.$$
\end{lemma}
\begin{proof}
We have 
$$\sin 2\pi z =\sum_{m=0}^\infty (-1)^m\frac{(2\pi z)^{2m+1}}{(2m+1)!}$$
and
$$q_f(z)= \sum_{m=0}^{[(w-6)/4]}  \frac{(-1)^{m}(2\pi z)^{2m+1}}{(2m+1)! } L_f(w-2m) +\frac{L_f\big(\frac{w+2}{2}\big) (2\pi z)^{\frac{w}2}}{2\big(\frac{w}{2}\big)!}.$$
Thus,
\begin{eqnarray*}
|\sin 2 \pi z -q_f(z)|&\le&\sum_{m\le w/8}\frac{(5\pi/2)^{2m+1}}{(2m+1)!} |L_f(w-2m)-1|\\
&&\qquad  +\sum_{w/8<m<w/4}\frac{(5\pi/2)^{2m+1}}{(2m+1)!} 
(|L_f(w-2m)|+1)\\
&&\qquad \qquad + \sum_{m>w/4}\frac{(5\pi/2)^{2m+1}}{(2m+1)!}\\ 
&=&\Sigma_1+\Sigma_2+\Sigma_3,
\end{eqnarray*} 
say. Now by Lemma \ref{lemma:est} we have
\begin{eqnarray*}
\Sigma_1 \le 4\times 2^{-k/4} \sum_{m\le w/8}\frac{(5\pi/2)^{2m+1}}{(2m+1)!}\le 4\times 2^{-k/4}\times  e^{5\pi/2}.
\end{eqnarray*}
We can combine estimates for $\Sigma_2$ and $\Sigma_3$. Again using Lemma \ref{lemma:est} we have
\begin{eqnarray*}
\Sigma_2 +\Sigma_3 \le (2\sqrt{k}\log 2k +2)\sum_{w/8<m}\frac{(5\pi/2)^{2m+1}}{(2m+1)!}.
\end{eqnarray*}
We can bound the sum using
\begin{eqnarray*}
\sum_{m=r+1}^\infty \frac{x^m}{m!}&=&\frac{x^{r+1}}{(r+1)!}\left(1+\frac{x}{r+2}+\frac{x^2}{(r+2)(r+3)}+\dots\right)\\
&\le& \frac{x^{r+1}}{(r+1)!}\frac{1}{(1-\frac{x}{r+1})}=\frac{x^{r+1}}{r!(r+1-x)}< \frac{(ex)^{r+1}}{r^r(r+1-x)},
\end{eqnarray*}
the last line uses $r!>(r/e)^r$.  Using this above we have 
\begin{eqnarray*}
\Sigma_2 +\Sigma_3 \le (2\sqrt{k}\log 2k +2)\frac{(5\pi e/2)^{k/2+1}}{(k/2)^{k/2}(k/2+1-5\pi/2)}.
\end{eqnarray*}
For $k\ge 80$ we have $\Sigma_1+\Sigma_2+\Sigma_3<0.01$.
\end{proof}

\begin{lemma}
If $k\ge 80$, then the function 
$$Q_f(z):=q_f(z)-q_f(1/z)$$
has at most 10 zeros in the annulus $A$.
\end{lemma}
This follows from Rouch\'{e}'s theorem using Lemmas \ref{lemma:S} and \ref{lemma:est2}.

\begin{corollary} 
If $k\ge 80$, then
$$\Im q_f(e^{i\theta})$$
has at most 10 zeros in $0\le \theta < 2\pi$.  Moreover, $\Im q_f(e^{i\theta})=0$ at $\theta=0$ and at $\theta=\pi$.
\end{corollary}
\begin{proof}
If $z$ is on the unit circle, then $\Im q_f(z) = -iQ_f(z)$ so any zero of $\Im q_f(z)$ has to be a zero of $Q_f(z)$. But $Q_f(z)$ has 
at most 10 zeros on the annulus $A$ of which the unit circle is a subset. Since $r_f(\pm1)=0$, we see from (\ref{eqn:q}) that 
$q_f(\pm 1)=0$, so $\Im q_f(e^{i\theta})=0$ for $\theta=0$ and $\theta=\pi$.
\end{proof}

  We now combine these lemmas to prove
  \begin{theorem} \label{thm80}
  
  Let $f$ be a cusp form of weight $k\geq 80$   for $\SL(2,\Z), w=k-2$  and  $p_f^-(z) = q_f(z)+z^wq_f(1/z)$ be its odd period polynomial of degree $w-1$. Then $p_f^-(z)$ has all  but 5 of its zeros on the unit circle. The 5 trivial zeros of $p_f(z)$are at $z=0, 2,-2,1/2,-1/2$. It has double zeros at $z=1,-1.$
  \end{theorem}

  \begin{proof}

  First recall that we have shown that each period polynomial has simple zeros at $z=0, 2,-2,1/2,-1/2$  and double zeros at $1,-1$. To prove that the rest of the zeros are on the unit circle we
  let $r(z)=q_f(z)$ and $N=w=k-2$ in  Lemma \ref{gen-lemma}.

  By the Corollary there are 10 zeros of $\Im(q_f(e^{i\th}))$ in  the interval $[0,2\pi)$ and hence by  part $b)$ of Lemma \ref{gen-lemma} for each of the $N-10$ intervals  among the N intervals $\cI_j= [\frac{\pi}{2M}+\frac{\pi}{M}j, \frac{\pi}{2M}+\frac{\pi}{M}(j+1)]$, $j=0, N-1$  for which $\Im(q_f(e^{i\th}))\neq 0$,
  $p_f(e^{i\th})=0$    for some
  $\th\in \cI_j$ This gives at least $N-10$ zeros on the unit circle. Among the 10 discarded intervals in which
  $\Im(q_f(e^{i\th})$  vanish,    we have also excluded the intervals that contain $\th=0$ and $\th=\pi$  where   $p_f(z)$ has double zeros. Hence we have at least $N-10+4=w-6$ zeros on the unit circle. Since the degree of $p_f(z)$ is $w-1$ together with the 5 zeros at   $z=0, 2,-2,1/2,-1/2$,  this covers all the zeros  and finishes the proof of the theorem.
    \end{proof}
    
    Finally   Theorem \ref{mainth} follows from Theorem \ref{thm80} together with the fact that for the  weights $k\leq 80$ the statement  can be verified numerically.

\end{document}